\documentclass[11pt, reqno]{amsart}
\usepackage{mathrsfs}
\usepackage[active]{srcltx}
\usepackage{mathrsfs,amsmath}
\usepackage{mathtools}
\usepackage{longtable}
\usepackage{todonotes}
\usepackage{amssymb}
\usepackage{cite}
\usepackage{float}
\usepackage[normalem]{ulem}
\allowdisplaybreaks
\usepackage{soul}

\usepackage{xcolor}
\definecolor{bluecite}{HTML}{0875b7}
\definecolor{Aranka}{RGB}{0, 135, 85}
\usepackage[left=3.2cm,right=3.2cm,top=3.2cm,bottom=3.2cm]{geometry}

\usepackage{xargs}                      
\newcommandx{\remg}[2][1=]{\todo[linecolor=green,backgroundcolor=green!25,bordercolor=green,#1]{#2}}
\newcommandx{\remb}[2][1=]{\todo[linecolor=blue,backgroundcolor=blue!25,bordercolor=blue,#1]{#2}}
\newcommandx{\remr}[2][1=]{\todo[linecolor=purple,backgroundcolor=purple!25,bordercolor=purple,#1]{#2}}
\usepackage[unicode=true,
bookmarksopen={true},
pdffitwindow=true,
colorlinks=true,
linkcolor=bluecite,
citecolor=bluecite,
urlcolor=bluecite,
hyperfootnotes=false,
pdfstartview={FitH},
pdfpagemode= UseNone]{hyperref}

\newtheorem{proposition}{Proposition}[section]
\newtheorem{theorem}{Theorem}[section]

\numberwithin{equation}{section}

\newcommand{\abs}[1]{\left|#1\right|}
\newcommand{\fel}{\frac{1}{2}}
\newcommand{\dz}{\mathrm{d}z}

\newcommand{\be}{\begin{equation}}
\newcommand{\ee}{\end{equation}}
\newcommand{\dl}{\mathrm{d}\lambda}

\newcommand\C{\mathbb C}
\newcommand\R{\mathbb R}
\newcommand\N{\mathbb N}
\newcommand\Z{\mathbb Z}
\newcommand\M{\mathbb M}

\newcommand{\T}{\mathbb{T}}
\newcommand{\Sn}{\mathbb{S}^n}

\newcommand{\Rn}{\mathbb{R}^{n+1}}

\newcommand{\id}{\mathrm{id}}
\newcommand{\Wpsn}{\mathcal{W}_p\left(\mathbb{S}^n,\|\cdot\|\right)}
\newcommand{\Wtsn}{\mathcal{W}_2\left(\mathbb{S}^n,\|\cdot\|\right)}

\newcommand\W{\mathcal W}

\newcommand\Wp{\mathcal{W}_p}
\newcommand\Wt{\mathcal{W}_2}
\newcommand\Wo{\mathcal{W}_1}
\newcommand\dwpp{d_{\W_p}^p}

\newcommand\dwt{d_{\Wt}}

\newcommand{\lers}[1]{\left\{ #1 \right\}}
\newcommand{\diam}{\mathrm{diam}}

\newcommand{\Wtsr}{\Wt\left(\Sn,\|\cdot\|\right)}

\newcommand{\ler}[1]{\left( #1 \right)}

\newcommand{\dd}{\mathrm{d}}

\newcommand{\inner}[2]{\left< #1,#2 \right>}

\newcommand{\potmu}{\mathcal{T}_\mu^{(p)}}

\newcommand{\potnu}{\mathcal{T}_\nu^p}

\newcommand{\supp}{\mathrm{supp}}
\newcommand{\affspan}{\mathrm{affspan}}

\title[Isometric rigidity of $\Wpsn$]{Isometric rigidity of Wasserstein spaces over Euclidean spheres}

\author[Gy.P. Geh\'er]{Gy\"orgy P\'al Geh\'er}
\address{Gy\"orgy P\'al Geh\'er, Department of Mathematics and Statistics\\ University of Reading\\ Whiteknights\\ P.O.
Box 220\\ Reading RG6 6AX\\ United Kingdom}
\email{gehergyuri@gmail.com}

\author[A. Hru\v{s}kov\'a]{Aranka Hru\v{s}kov\'a}
\address{Aranka Hru\v{s}kov\'a, HUN-REN Alfr\'ed R\'enyi Institute of Mathematics\\ Re\'altanoda u. 13-15.\\
Budapest 1053\\ Hungary\\ and Central European University, Nádor u. 9, Budapest 1051, Hungary}
\email{umim.cist@gmail.com, aranka@berkeley.edu}

\author[T. Titkos]{Tam\'as Titkos}
\address{Tam\'as Titkos, Corvinus University of Budapest, Department of Mathematics \\
Fővám tér 13-15 \\ Budapest 1093 \\ Hungary\\ and HUN-REN Alfr\'ed R\'enyi Institute of Mathematics\\ Re\'altanoda u. 13-15.\\
Budapest 1053\\ Hungary}

\email{titkos@renyi.hu}

\author[D. Virosztek]{D\'aniel Virosztek}
\address{D\'aniel Virosztek, HUN-REN Alfr\'ed R\'enyi Institute of Mathematics\\ Re\'altanoda u. 13-15.\\
Budapest 1053\\ Hungary}

\email{virosztek.daniel@renyi.hu}

\begin{document}


\thanks{Geh\'er was supported by the Leverhulme Trust Early Career Fellowship (ECF-2018-125), and also by the Hungarian National Research, Development and Innovation
Office (Grant no. K115383); Hrušková was supported by the DYNASNET ERC Synergy grant, agreement ID 810115 and by the NKFIH KKP 139502 project; Titkos was supported by the Hungarian National Research, Development and Innovation Office - NKFIH (grant no. K115383) and by the Momentum Program of the Hungarian Academy of Sciences (grant no. LP2021-15/2021); Virosztek was supported by the Momentum Program of the Hungarian Academy of Sciences (grant no. LP2021-15/2021)
and partially supported by the ERC Consolidator Grant no. 772466.}

\maketitle

\begin{abstract} 
We study the structure of isometries of the quadratic Wasserstein space $\Wtsn$ over the sphere endowed with the distance inherited from the norm of $\mathbb{R}^{n+1}$. We prove that $\Wtsn$ is isometrically rigid, meaning that its isometry group is isomorphic to that of $\ler{\Sn,\|\cdot\|}$. This is in striking contrast to the non-rigidity of its ambient space $\W_2\ler{\Rn,\|\cdot\|}$ but in line with the rigidity of the geodesic space $\W_2\ler{\Sn,\sphericalangle}$. One of the key steps of the proof is the use of mean squared error functions to mimic displacement interpolation in $\Wtsn$. A major difficulty in proving rigidity for quadratic Wasserstein spaces is that one cannot use the Wasserstein potential technique. To illustrate its general power, we use it to prove isometric rigidity of $\W_p\left(\mathbb{S}^1, \|\cdot\|\right)$ for $1 \leq p<2$.
\end{abstract}

\tableofcontents

\section{Motivation and main result}

In recent years, there has been considerable activity in characterising isometries of various metric spaces of probability measures. See e.g. \cite{btv1,btv2,bertrand-kloeckner-2016,DolinarKuzma,DolinarMolnar,Kuiper,KS-Molnar,Levy,LP,JMAA,MolnarSzokol,TAMS,HIL,TnSn,exotic,RIMS,Kloeckner-2010,S-R,titkoskissgraf,Virosztek} for results about the total variation, L\'evy, Kuiper, L\'evy-Prokhorov, Kolmogorov-Smirnov, and Wasserstein metrics. Among these, an interesting result is due to Kloeckner. In \cite[Theorem 1.1 and Theorem 1.2]{Kloeckner-2010}, he shows that the quadratic Wasserstein space $\Wt\ler{\Rn,\|\cdot\|}$, where $\|\cdot\|$ stands for the metric induced by the norm, exhibits the rare phenomenon of not being isometrically rigid, meaning that not all isometries of $\Wt\ler{\Rn,\|\cdot\|}$ are induced by an isometry of $\ler{\Rn,\|\cdot\|}$. In this paper, we consider the metric subspace $\ler{\mathbb{S}^{n},\|\cdot\|}$ of the base space $\ler{\Rn,\|\cdot\|}$ and prove that the non-rigidity does not carry over: the exotic isometries of $\Wt\ler{\Rn,\|\cdot\|}$ send measures supported on $\mathbb{S}^{n}$ to measures supported also outside of $\mathbb{S}^{n}$, while we gain no new exotic isometries by restricting to this smaller metric space. In general, when $H$ is an arbitrary Borel subset of $\R^{n+1}$, then $\Wp(H,\|\cdot\|)$ embeds isometrically into $\Wp\left(\R^{n+1},\|\cdot\|\right)$, but this does not necessarily imply that there exists such a natural embedding for their isometry groups. To see an example, we mention the case of the real line $(\R,|\cdot|)$ with the subset $H=[0,1]$ (see \cite[Theorem 2.5 and Theorem 3.7]{TAMS} for details): the isometry group of $\Wo([0,1],|\cdot|)$ is the Klein group, which cannot be embedded by a group homomorphism into the isometry group of $\mathcal{W}_1(\R,|\cdot|)$, which is isomorphic to the isometry group of the real line.

Finally, we draw attention to Santos-Rodr\'iguez's paper \cite{S-R} and our recent work \cite{TnSn}. In \cite{S-R}, the author considers (among others) Wasserstein spaces with $p>1$ whose underlying metric space is a rank-one symmetric space, which class contains the sphere $\Sn$ with the spherical distance $\sphericalangle$, while in \cite{TnSn}, we considered finite-dimensional tori and spheres with their geodesic distances for all parameters $p \geq 1$. Together, these two papers show that $\Wp\ler{\Sn,\sphericalangle}$ is isometrically rigid for all $p\geq1$. As explained above, in this paper, we replace the angular distance $\sphericalangle$ with another natural metric: the distance inherited from the norm of $\mathbb{R}^{n+1}$. We focus on the case of $p=2$ because this is the only parameter value for which the ambient space $\W_p\ler{\Rn, \|\cdot\|}$ is \emph{not} rigid. We expect that for $p \neq 2$, techniques similar to the ones used in \cite{HIL} would lead to a proof of isometric rigidity. The situation is analogous to the case of the real line and the unit interval: the quadratic Wasserstein space is not rigid over $\R$ but it is rigid over the compact subset $[0,1]$, see \cite[Theorem 1.1]{Kloeckner-2010} and \cite[Theorem 2.6]{TAMS}. Our main result reads as follows.

\begin{theorem}\label{thm:main}
For all $n\in\N$, the quadratic Wasserstein space $\Wtsr$ is isometrically rigid. That is, for any isometry $\Psi\colon\Wtsr\to\Wtsr$, there exists an isometry $\psi\colon\Sn\to\Sn$ such that $\Psi = \psi_\#$.
\end{theorem}

In our recent works \cite{HIL,TnSn}, recovering measures from their Wasserstein potentials --- see \eqref{eq:pot-def} for precise definition --- turned out to be a powerful method to prove isometric rigidity.
However, this method cannot be used in the case of $\Wtsn$, as shown by the following simple example. Let $\delta_x$ denote the Dirac measure concentrated at $x\in\mathbb{S}^n$, let $\mu_z:=\frac{1}{2}(\delta_z +\delta_{-z})$ for $z \in \mathbb{S}^n$, and note that for any $x \in \mathbb{S}^n$ we have $d_{\W_2}^2(\mu_z,\delta_x)=\frac{1}{2}(\|x-z\|^2+\|x+z\|^2)=2$ independently of $x$ and $z$ --- see \eqref{eq:wasser_def} for the precise definition of the $p$-Wasserstein distance $d_{\Wp}$. This means that every element of the set $\lers{\mu_z \, \middle| \, z \in \mathbb{S}^n}$ has the same Wasserstein potential function, and hence potentials do not determine measures uniquely in general.
\par
Our complimentary result Theorem \ref{thm:circle} demonstrates sensitivity of the Wasserstein potential method to the parameter value $p$. Namely, we show that, at least in the special case of $\mathbb{S}^1$, measures are uniquely determined by their potentials if $1 \leq p <2$, and hence $\mathcal{W}_p\left(\mathbb{S}^1,\|\cdot\|\right)$ is isometrically rigid.

\section{The Wasserstein space $\Wpsn$ and the Wasserstein potential}

In this section, we recall all the necessary notions and notations. 
Let $(Y,\rho)$ be a complete and separable metric space, $p\geq1$ a fixed real number, and $\mathcal{P}(Y)$ the set of all Borel probability measures on $Y$. The $p$-Wasserstein space $\Wp(Y,\rho)$, where $p\in[1,\infty)$, is then defined as the set 
\begin{equation*}
\left\{\mu\in\mathcal{P}(Y)\,\Bigg|\,\exists \hat{y}\in Y:~~\int\limits_Y \rho(y,\hat{y})^p~\mathrm{d}\mu(y)<\infty\right\}
\end{equation*}
of probability measures endowed with the \emph{$p$-Wasserstein metric}
\begin{equation}
 \label{eq:wasser_def}
d_{\W_p}(\mu, \nu):=\ler{\inf_{\pi \in \Pi(\mu, \nu)} \iint\limits_{Y \times Y} \rho(x,y)^p~\dd \pi(x,y)}^{1/p},
\end{equation}
where the infimum is taken over the set $\Pi(\mu,\nu)$ of all couplings of $\mu$ and $\nu$. A Borel probability measure $\pi$ on $Y \times Y$ is called a \emph{coupling} of $\mu$ and $\nu$ if $\pi\ler{A \times Y}=\mu(A)$ and $\pi\ler{Y \times B}=\nu(B)$ for all Borel sets $A,B\subseteq Y$. For more details about Wasserstein spaces, we refer the reader to the comprehensive textbooks \cite{AG,Fb,Santambrogio,Villani}. Now we only mention that optimal couplings always exist, and the infimum in (\ref{eq:wasser_def}) becomes minimum \cite[Theorem 1.5]{AG}. Furthermore, finitely supported measures are dense in Wasserstein spaces, see, e.g., \cite[Theorem 6.18]{Villani}. 

An \emph{isometric embedding} between metric spaces $(X,d)$ and $(Y,\rho)$ is a map $\phi\colon(X,d)\to(Y,\rho)$ which preserves distances, i.e., a map such that $d(x,x')=\rho\left(\phi(x),\phi(x')\right)$ for all $x,x'\in X$. We shall use the term \emph{isometry} for a surjective isometric embedding from a metric space onto itself. It is important to note that if $(X,d)$ is a compact metric space, then every isometric embedding $\phi\colon(X,d)\to(X,d)$ is surjective and hence an isometry \cite[Theorem 1.6.14]{MG-book}.

For a Borel-measurable map $\psi\colon Y \rightarrow Y$, its push-forward $\psi_\# \colon\mathcal{W}_p(Y,\rho)\to\mathcal{W}_p(Y,\rho)$ is defined by $\big(\psi_\#(\mu)\big)(A):=\mu(\psi^{-1}[A])$, where $A\subseteq Y$ is a Borel set and $\psi^{-1}[A]=\left\{x\in X\,|\,\psi(x)\in A\right\}$.
In particular, when $\psi\colon Y\to Y$ is an isometry, then so is $\psi_\#$ by the very definition of the Wasserstein distance, giving rise to a canonical embedding of the isometries of $(Y,\rho)$ to the isometries of $\mathcal{W}_p(Y,\rho)$.

In this paper, we consider the compact metric space $\left(\Sn,\|\cdot\|\right)$, where
\[
\Sn := \{x\in\R^{n+1}\colon \|x\|=1\}
\]
is the unit sphere of $\mathbb{R}^{n+1}$
equipped with the distance inherited from the Euclidean norm of $\R^{n+1}$, that is, $\mathrm{dist}(x,y)=\|x-y\|$. The point $-x$ is called the antipodal of $x$. Since $\left(\Sn,\|\cdot\|\right)$ is bounded, the Wasserstein space $\Wpsn$ is the entire set $\mathcal{P}\left(\Sn\right)$ endowed with the distance
\begin{equation} \label{eq: wasser def}
d_{\W_p}(\mu, \nu):=\ler{\inf_{\pi \in \Pi(\mu, \nu)} \iint\limits_{\Sn \times \Sn} \|x-y\|^p~\dd \pi(x,y)}^{1/p}.
\end{equation}  
We write $\Wp\ler{\Sn,\|\cdot\|}$ instead of the usual $\Wp\ler{\Sn}$ notation to avoid any confusion with the results in \cite{TnSn,S-R}. As the Wasserstein distance metrizes the weak convergence of probability measures over bounded metric spaces (see, e.g., \cite[Theorem 7.12]{Villani-03}), by Prokhorov's theorem, $\left(\Sn,\|\cdot\|\right)$ being compact tells us that $\Wp(\Sn,\|\cdot\|)$ is compact too --- see also Remark 6.19 in \cite{Villani}. This implies that every isometric embedding of $\Wp\ler{\Sn,\|\cdot\|}$ into itself is an isometry.
\par
For a measure $\mu\in\mathcal{P}\left(\Sn\right)$, its support $\supp(\mu)$ is the set of all points $x\in\Sn$ for which every open neighbourhood of $x$ has positive measure. As usual, $\delta_x$ denotes the Dirac measure supported on the single point $x\in\Sn$.\\

The question arises whether it is possible to identify a measure if we know its distance from all Dirac measures. (Recall that $d_{\W_p}(\delta_x,\delta_y)=\|x-y\|$ for all $x,y\in\Sn$ and thus the set of all Dirac measures is an isometric copy of the underlying metric space.) To answer this question, we first introduce the notion of \emph{Wasserstein potential} $\potmu$. For a given $\mu\in\Wpsn$, the Wasserstein potential is the function
\begin{equation}\label{eq:pot-def}
\potmu\colon\Sn\to\mathbb{R};\qquad	\potmu(x):=d_{\W_p}^p\left(\delta_x,\mu\right)=\int\limits_{\Sn} \|x-y\|^p ~\dd\mu(y).
\end{equation}

Now, the question above can be rephrased as follows: \emph{does the Wasserstein potential determine the measure uniquely?}

\section{Does the Wasserstein potential determine the measure uniquely?}
The answer to this question is no, in general. A prominent example is $\Wtsn$ where measures supported on antipodal points with both weights equal to $\frac{1}{2}$ have the same (constant) potential function --- see the example in Section 1, after Theorem \ref{thm:main}. Beyond this, exotic isometries of $\Wt(\R,|\cdot|)$ (see \cite[Section 5.1 and Section 5.2]{Kloeckner-2010}) are also counterexamples. 
\par 
However, we will now prove that it does in the case of $\mathbb{S}^1\simeq\T=\lers{z \in \C \, : \, \abs{z}=1}$ equipped with the distance function $r(z, \omega)=\abs{\fel (z-\omega)}$ for $1\leq p<2$. This normalization of the distance is consistent with the one used in \cite{Virosztek}. In this section, we assume that $1\leq p<2$, and we recall that the $p$-Wasserstein distance of $\mu, \nu \in \Wp(\T)$ in this case is  
\[
d_{\W_p}\ler{\mu,\nu}=\left(\inf_{\pi \in \Pi\ler{\mu, \nu}} \iint\limits_{\T \times \T} \abs{\fel (z-\omega)}^p~\dd\pi\ler{z,\omega}\right)^{1/p},
\]
and therefore for any $z\in\mathbb{T}$, the Wasserstein potential is of the form
\be \label{eq:dist-dirac}
\potmu(z)=d_{\W_p}^p\ler{\delta_z, \mu}=\int\limits_\T \abs{\fel (z-\omega)}^p ~\dd\mu\ler{\omega}.
\ee

We showed in \cite{TnSn} that Fourier analytic methods can sometimes solve the problem of rigidity in a very elegant way. For example, we showed that isometric rigidity of $\mathcal{W}_2(\mathbb{T},\sphericalangle)$ can be proved by using the Fourier transform of the Wasserstein potential, however, the same method fails in the case $\mathcal{W}_1(\mathbb{T},\sphericalangle)$. As we will see, if we endow $\mathbb{T}$ with the distance $r(z,\omega)=\big|\frac{1}{2}(z-\omega)\big|$, then the situation changes: the same method works to prove isometric rigidity of $\mathcal{W}_1(\mathbb{T},r)$, but fails in the case $\mathcal{W}_2(\mathbb{T},r)$.

Now we recall the very basics of Fourier analysis on the abelian group $\T$. The main reason for doing so is to fix the notation. The continuous \emph{characters} of $\T$ are exactly the power functions with an integer exponent. That is, if $\varphi_k(z)=z^k$ for all $k\in\mathbb{Z}$ and $\Gamma$ is the \emph{dual group} (i.e., the group of all continuous characters), then $\Gamma=\lers{\varphi_k\colon\mathbb{T}\to\mathbb{C} \, | \, k \in \Z}$, and $\Gamma \cong \Z$. The group $\T$ is compact, hence it admits a unique \emph{Haar probability measure} $\lambda$, which can be expressed explicitly as
\[
\dd\lambda\ler{z}=\frac{\dz}{2 \pi i z }.
\]
The \emph{Fourier transform} of a (complex-valued) function $f \in L^1\ler{\T, \lambda}$ is defined by
\[
\hat{f}(k)=\int\limits_\T f \overline{\varphi_k} \dl=\frac{1}{2 \pi i} \int\limits_\T f(z) z^{-(k+1)} \dz\qquad(k\in\mathbb{Z}).
\]
Let us denote the set of all (complex-valued) measures of finite total variation by $\M\ler{\T}$. The Fourier transform of $\mu \in \M\ler{\T}$ is defined by
\[
\hat{\mu}(k)=\int\limits_\T \overline{\varphi_n} \dd \mu =\int\limits_\T z^{-k}~\dd\mu\ler{z}\qquad(k\in\mathbb{Z}).
\]
We note that $L^1$ functions can be naturally identified with absolutely continuous measures (with respect to the Haar measure), see \cite[Subsection 1.3.4.]{rudin}.
The convolution of $L^1$ functions $f$ and $g$ is defined by
\be
\ler{f*g}(z)=\int\limits_\T f\left(z \omega^{-1}\right)g(\omega)~\dd\lambda(\omega)
\ee
and the convolution of $f \in L^1\ler{\T, \lambda}$ and $\mu \in \M(\T)$ is defined by
\be \label{eq:meas-conv-def}
\ler{f*\mu}(z)=\int\limits_\T f\ler{z \omega^{-1}}~\dd\mu(\omega).
\ee
It is a key identity that the Fourier transform factorizes the convolution, that is,
\be \label{eq:f-key-id}
\widehat{f * \nu}=\hat{f}\cdot \hat{\nu}.
\ee

Now we are ready to state and prove the main result of this section. It says that if $1\leq p <2$, then the Wasserstein space $\Wp(\mathbb{T},r)$ is isometrically rigid.

\begin{theorem} \label{thm:circle}
Let $p\in[1,2)$ be a real number and let $\Psi\colon \Wp\ler{\mathbb{T},r} \rightarrow \Wp\ler{\mathbb{T},r}$ be an isometry. Then there exists an isometry $\psi\colon\ler{\mathbb{T},r} \rightarrow\ler{\mathbb{T},r}$ such that $\Psi=\psi_\#$.
\end{theorem}

\begin{proof} First observe that the diameter of $\Wp(\mathbb{T},r)$ is 1, and $d_{\W_p}(\mu,\nu)=1$ if and only if $\mu=\delta_x$ and $\nu=\delta_{-x}$ for some $x\in\mathbb{T}$. Since $\Psi$ is an isometry, we have
\begin{equation}\label{diracfix}
    1=d_{\W_p}(\delta_x,\delta_{-x})=d_{\W_p}\left(\Psi(\delta_x),\Psi(\delta_{-x})\right)
\end{equation}
for all $x\in\mathbb{T}$, which implies that $\Psi\left(\delta_x\right)$ is a Dirac measure as well.

Let us define the map $\psi\colon\mathbb{T}\to\mathbb{T}$ via the identity $\Psi(\delta_x)=\delta_{\psi(x)}$ -- this means that $\Psi$ coincides with $\psi_\#$ on the set of Dirac measures. The map $\psi\colon(\mathbb{T},r)\to(\mathbb{T},r)$ is in fact an isometry: 
\[
r\ler{\psi(x),\psi(y)}=d_{\W_p}\ler{\delta_{\psi(x)},\delta_{\psi(y)}}
=d_{\W_p}\ler{\Psi(\delta_x),\Psi(\delta_y)}
=d_{\W_p}\ler{\delta_x,\delta_y}=r(x,y)
\]
for all $x,y\in\mathbb{T}$, and $(\mathbb{T},r)$ is compact. These together combine into that $\ler{\psi^{-1}}_\#\circ\Psi$ is an isometry which fixes all Dirac measures.
If we now prove that any isometry of $\mathcal{W}_p(\mathbb{T},r)$ which fixes all Dirac measures must be the identity, we are done: in that case, $\ler{\psi^{-1}}_\#\circ\Psi=\id_{\mathcal{W}_p(\mathbb{T},r)}$, i.e., $\Psi=\psi_\#$ as claimed.

From now on, let us assume that $\Phi\colon\mathcal{W}_p(\mathbb{T},r)\to\mathcal{W}_p(\mathbb{T},r)$ is an isometry such that $\Phi(\delta_z)=\delta_z$ for all $z\in\mathbb{T}$. Then we have
\[
\potmu(z)=\dwpp\ler{\delta_z, \mu}=\dwpp\ler{\Phi(\delta_z),\Phi(\mu)}=\dwpp\ler{\delta_z, \Phi\ler{\mu}}=\mathcal{T}^p_{\Phi(\mu)}(z)
\]
for all $z \in \mathbb{T}$ and $\mu \in \Wp\ler{\mathbb{T},r}$. The question is whether this implies $\mu=\Phi(\mu)$. The proof will be done once we prove that a measure $\mu\in\Wp(\mathbb{T},r)$ is uniquely determined by its Wasserstein potential. To this end, assume that $\mu$ and $\nu$ are two measures such that
\begin{equation}\label{pmupnu}
\potmu(z)=\potnu(z)\quad\mbox{for all}\ \ z\in\mathbb{T}.
\end{equation}
We need to show that \eqref{pmupnu} implies $\mu=\nu$. Let us introduce the map 
\begin{equation}\label{fp} 
f_p(z):=\abs{\fel (z-1)}^p.
\end{equation}
Then by \eqref{eq:dist-dirac} and \eqref{eq:meas-conv-def} one can observe that $\potmu(z)=(f_p*\mu)(z)$ holds for all $z \in \T$ and $\mu \in \Wp(\T)$. Indeed, we have
\begin{equation} \label{eq:dist-conv}
\begin{split}
\potmu(z)&=\dwpp\ler{\delta_z, \mu}=\int\limits_\T \abs{\fel (z-\omega)}^p~\dd\mu\ler{\omega}\\
&=\int\limits_\T \abs{\fel \left(z\omega^{-1}-1\right)}^p~\dd\mu\ler{\omega}
=\int\limits_\T f_p\ler{z \omega^{-1}}~\dd\mu\ler{\omega}=\ler{f_p*\mu}(z).
\end{split}
\end{equation}
The key observation is that the Fourier transform of $f_p$ does not vanish anywhere, that is, $\hat{f_p}(n)\neq 0$ for all $n \in \Z$.
For $n=0$, we have
\[
\hat{f_p}(0)=\int\limits_\T \abs{\fel (z-1)}^p~\dd\lambda(z)>0,
\]
while for $n \neq 0$, we use that
\[
f_p(z)=\ler{\abs{\fel (z-1)}^2}^{\frac{p}{2}}=\ler{\frac{1}{4}\ler{2-z-z^{-1}}}^{\frac{p}{2}}=\ler{1-\frac{1}{4}\ler{2+z+z^{-1}}}^{\frac{p}{2}},
\]
and by the binomial series expansion we get that
\be \label{eq:binom-expansion}
f_p(z)=\sum_{k=0}^{\infty} \binom{\frac{p}{2}}{k} \ler{\frac{-1}{4}\ler{2+z+z^{-1}}}^k,
\ee
where $\binom{\frac{p}{2}}{0}=1$ and $\binom{\frac{p}{2}}{k}=\frac{\prod_{j=0}^{k-1}\left(\frac{p}{2}-j\right)}{k!}$. Using that the sign of $\binom{\frac{p}{2}}{k} \ler{-1}^k$ is negative for all $k\geq 1$, equality \eqref{eq:binom-expansion} can be written as
\be\label{eq:binom-atalakitott}
f_p(z)=1-\left\{\frac{p}{2}\cdot\frac{2+z+z^{-1}}{4}+\sum_{k=2}^{\infty}\Bigg(\frac{\frac{p}{2}\prod_{j=1}^{k-1}(j-\frac{p}{2})}{k!}\Big(\frac{2+z+z^{-1}}{4}\Big)^k\Bigg)\right\}.
\ee
It is a useful feature of the group $\T$ that the Fourier series of a function coincides with its power series. Therefore, the above binomial expansion gives us useful information about $\hat{f_p}$, namely, $\hat{f_p}(k)$ coincides with the coefficient of $z^k$ in the expansion \eqref{eq:binom-expansion}. 

Let us note that for $n \neq 0$, the coefficient of $z^n$ must be strictly negative because the expressions $\frac{p}{2}, 1-\frac{p}{2},2-\frac{p}{2}, \dots$ are all positive -- here we use the assumption that $p<2$. So we obtained that $\hat{f_p}(0)>0$ and $\hat{f_p}(n)<0$ for $n \neq 0$ which means that $\hat{f_p}(n)\neq 0$ for all $n \in \Z$.
\par
By \eqref{eq:dist-conv}, the assumption that $\potmu(z)=\potnu(z)$ for all $z\in\mathbb{T}$ implies that $f_p * \mu=f_p * \nu$. By \eqref{eq:f-key-id}, this means that $\hat{f_p} \cdot \hat{\mu}=\hat{f_p} \cdot \hat{\nu}$. Since $\hat{f_p}(n)\neq 0$ for every $n$, we can deduce that $\hat{\mu}=\hat{\nu}$, but the Fourier transform completely determines the measure \cite[Chapter 1]{rudin}, hence $\mu=\nu$, and the proof is done.

\end{proof}

\section{Isometric rigidity of $\Wtsr$ --- proof of Theorem \ref{thm:main}}

The assumption $p<2$ was crucial in the previous section, and therefore the quadratic case cannot be handled with the same Fourier-analytic technique. In this section, we use a method that allows us to prove isometric rigidity in the quadratic case not only over the circle but over higher-dimensional spheres too. We start this section with three propositions which will be utilized later in the proof of Theorem \ref{thm:main}.
\par 
The first proposition, which can be found also in the Appendix of \cite{HIL} (see the proof of Lemma 3.13 there), helps us understand how a translation affects the Wasserstein distance. For $\mu\in\mathcal{P}\left(\Rn\right)$ and $v\in \Rn$, the translation of $\mu$ by $v$ is the measure $(t_v)_\#\mu\in\mathcal{P}\left(\Rn\right)$ where $t_v\colon \Rn\to \Rn$, $x\mapsto x+v$ is the translation by $v$. Recall that elements of $\Wtsr$ can be considered as elements of $\Wt(\mathbb{R}^{n+1},\|\cdot\|)$ because $\Wt(\mathbb{S}^n,\|\cdot\|)$ naturally embeds into $\Wt(\mathbb{R}^{n+1},\|\cdot\|)$. The barycenter of $\mu\in\Wtsr$ is defined to be the point $m(\mu)=\int_{\mathbb{S}^n} x~\mathrm{d}\mu(x)\in\mathbb{R}^{n+1}$. Note that the barycenter $m(\mu)$ defined above is also the unique point in $\R^{n+1}$ satisfying $\inner{m(\mu)}{z}=\int_{\R^{n+1}} \inner{x}{z} \dd \mu(x)$ for all $z \in \R^{n+1}$, which equation is used to define the barycenter in an infinite-dimensional setting. However, in finite dimensions, the above direct definition, which does not refer to the dual space, is available.

\begin{proposition}\label{prop:translation}
Let $\mu,\nu\in\Wt\left(\Rn\right)$ and $v\in\Rn$. Then we have
	\begin{equation}\label{eq:transl}
	\dwt^2\left((t_v)_\#\mu,\nu\right) = \dwt^2(\mu,\nu) + \left\langle v, v + 2m(\mu) - 2m(\nu) \right\rangle.
	\end{equation}
	In particular, substituting $v = m(\nu) - m(\mu)$ gives
	\begin{equation}\label{eq:bari-transl} \dwt^2(\mu,\nu) = \dwt^2\left((t_{- m(\mu)})_\#\mu,(t_{- m(\nu)})_\#\nu\right) + \|m(\nu) - m(\mu)\|^2. \end{equation}
		Subsequently, $\nu$ is a translated version of $\mu$ if and only if $\dwt(\mu,\nu) = \|m(\nu) - m(\mu)\|$.
\end{proposition}
\begin{proof}
	For any $\pi\in\Pi(\mu,\nu)$ and $v\in\Rn$, we have $(t_{(v,0)})_\#\pi\in\Pi\left((t_v)_\#\mu,\nu\right)$, and vice versa. (Here, $0$ stands for $0 \in \R^{n+1}$.) Hence
	\begin{align*}
	\dwt^2\left((t_v)_\#\mu,\nu\right) &= \inf_{\pi \in \Pi(\mu, \nu)} \iint_{\Rn \times \Rn} \|x-y\|^2~\dd \ler{(t_{(v,0)})_\#\pi}(x,y) \\
	&= \inf_{\pi \in \Pi(\mu, \nu)} \iint_{\Rn \times \Rn} \|x+v-y\|^2~\dd \pi(x,y)\\
	&= \inf_{\pi \in \Pi(\mu, \nu)} \iint_{\Rn \times \Rn} \ler{\|x-y\|^2+\|v\|^2+2\inner{x}{v}-2\inner{y}{v}}~\dd \pi(x,y) \\
	&= \dwt^2\left(\mu,\nu\right) + \|v\|^2 + 2\int_{\Rn} \inner{x}{v}~\dd \mu(x) -2\int_{\Rn}\inner{y}{v}~\dd \nu(y),
	\end{align*}
	which gives \eqref{eq:transl}. The identity \eqref{eq:bari-transl} follows if we translate both arguments in the left-hand side by the vector $m(\nu)$.
\end{proof}

In quadratic Wasserstein spaces over uniquely geodesic spaces, the $\alpha$-weighted mean squared-error function
\[
\rho \mapsto (1-\alpha)d_{\mathcal{W}_2}^2(\mu,\rho)+\alpha d_{\mathcal{W}_2}^2(\nu, \rho)
\]
defined by $\mu$ and $\nu$ has a unique minimizer --- provided that the optimal coupling of $\mu$ and $\nu$ is unique --- which is the \emph{displacement convex combination} or \emph{displacement interpolation} of $\mu$ and $\nu$ with weights $(1-\alpha)$ and $\alpha$ \cite{Villani-03,Villani}.
Intuitively, this is the measure that we obtain if we start moving $\mu$ to $\nu$ according to the optimal transport plan, but stop at proportion $\alpha$ of the journey.
A great challenge concerning $\ler{\mathbb{S}^n, \|\cdot\|}$ is that it has no geodesics at all, and hence the quadratic Wasserstein space $\Wtsn$ has no geodesics either. Still, mean squared-error functions make perfect sense on $\Wtsn$, they are invariant under isometries in an appropriate sense, and hence if the measures $\mu$ and $\nu$ defining them are fixed by an isometry $\Phi$, then so are the unique minimizers --- if they exist.
We will prove in Proposition \ref{prop:min-char} that on $\ler{\Sn,\|\cdot\|}$, the minimizer of the $\alpha$-weighted squared-error function is the projection of the displacement interpolation onto the sphere. This is similar to how for a measure $\mu\in\mathcal{P}\ler{\Sn}$, its closest Dirac measure supported on a point in $\Rn$ is $\delta_{m(\mu)}$, while among those supported on $\Sn$, it is the projection of $\delta_{m(\mu)}$.
We are going to exploit this characterisation in Step 6 of the proof of Theorem \ref{thm:main}.

We will use the following projection of the $\alpha$-weighted mean of two points $x,y$ onto $\mathbb{S}^n$ frequently:

\begin{equation*}\label{eq:p-alpha-def}
p_{\alpha}(x,y):=\frac{(1-\alpha)x+\alpha y}{\|(1-\alpha)x+\alpha y\|} \qquad \ler{\alpha \in [0,1], \, x,y \in \mathbb{S}^n}.  
\end{equation*} 
Note that $p_\alpha(x,y)$ is not defined when $\alpha=\frac{1}{2}$ and $x=-y$.

Let us define the cost $c_\alpha\colon \mathbb{S}^n \times \mathbb{S}^n \rightarrow [0,2]$ by
\be \label{eq:c-alpha-def}
c_{\alpha}(x,y):=\min_{z \in \Sn} \lers{(1-\alpha)\|x-z\|^2+\alpha \|z-y\|^2} =2(1-\|(1-\alpha)x+\alpha y\|).
\ee
If $\pi \in \mathcal{P}\left(\mathbb{S}^n \times \mathbb{S}^n\right)$ is a coupling of $\mu$ and $\nu$, then let $\rho_\pi^{(\alpha)} \in \Wtsn$ be defined to be the displacement interpolation in $\R^{n+1}$ at time $\alpha$ between $\mu$ and $\nu$ according to the plan $\pi$, projected to $\Sn$. Formally,
\be \label{eq:rho-pi-def}
\rho_{\pi}^{(\alpha)}:=\ler{\lers{(x,y) \mapsto p_{\alpha}(x,y)}}_{\#} \pi.
\ee
Now let $\mu, \nu \in \Wtsn$ and consider the $\alpha$-weighted mean squared error
\be \label{eq:mse-def}
Q_{\alpha}^{\mu, \nu}\colon \, \Wtsn \rightarrow [0,\infty); \,  \rho \mapsto Q_{\alpha}^{\mu, \nu}(\rho):= (1-\alpha) d_{\mathcal{W}_2}^2(\mu, \rho)+\alpha d_{\mathcal{W}_2}^2(\nu, \rho).
\ee

\begin{proposition} \label{prop:min-char}
Let $\mu, \nu \in \Wtsn$ be such that there is a unique optimal transport plan $\pi_{*}$ for them with respect to the cost $c_\alpha$ defined in \eqref{eq:c-alpha-def}.
Suppose that $\alpha\neq\frac{1}{2}$ or $\alpha=\frac{1}{2}$ and $\pi_*\ler{\lers{(z,-z) : z\in \mathbb{S}^n}}=0$.
Then the mean squared error $Q_{\alpha}^{\mu, \nu}$ defined in \eqref{eq:mse-def} has a unique minimizer which is equal to $\rho_{\pi_{*}}^{(\alpha)}$, the push-forward of $\pi_*$ by $p_\alpha$ --- see \eqref{eq:rho-pi-def} for the precise definition. If $\pi_*\ler{\lers{(z,-z) : z\in \mathbb{S}^n}}>0$ and $\alpha=\frac{1}{2}$, then $\rho_{\pi_*}^{(\alpha)}=\rho_{\pi_*}^{(\frac{1}{2})}$ is not well-defined and $Q_\alpha^{\mu,\nu}=Q_{\frac{1}{2}}^{\mu,\nu}$ has infinitely many minimizers.
\end{proposition}

\begin{proof}
We proceed by establishing a lower bound for \eqref{eq:mse-def} and taking care of the case of equality. Let $\rho \in \Wtsn$ be arbitrary, and let $\pi_{\mu, \rho}$ and $\pi_{\rho,\nu}$ be optimal transport plans (w.r.t. the quadratic distance) between $\mu$ and $\rho$, and $\rho$ and $\nu$, respectively. Let $\pi_{\mu,\rho,\nu} \in \mathcal{P}\left(\Sn \times \Sn \times \Sn\right)$ be the \emph{gluing} of $\pi_{\mu, \rho}$ and $\pi_{\rho,\nu}$ --- see \cite[Lemma 7.6]{Villani-03} for the precise definition. Then $\pi_{\mu,\nu}:=(\pi_{\mu,\rho,\nu})_{1,3} \in \mathcal{P}\left(\Sn \times \Sn\right)$ is a coupling of $\mu$ and $\nu$. Now
\begin{equation}\label{eq:lb-2}
\begin{split}
Q_{\alpha}^{\mu, \nu}(\rho)&=(1-\alpha) d_{\mathcal{W}_2}^2(\mu, \rho)+\alpha d_{\mathcal{W}_2}^2(\nu, \rho)\\
&=(1-\alpha) \iint_{\Sn \times \Sn} \|x-z\|^2\, \dd \pi_{\mu, \rho}(x,z)+
\alpha \iint_{\Sn \times \Sn} \|z-y\|^2\, \dd \pi_{\rho, \nu}(z,y)\\
&=\iiint_{\Sn \times \Sn \times \Sn} (1-\alpha) \|x-z\|^2+\alpha \|z-y\|^2\, \dd \pi_{\mu,\rho,\nu}(x,z,y)\\
&\geq\iiint_{\Sn \times \Sn \times \Sn} c_{\alpha}(x,y)\, \dd \pi_{\mu,\rho,\nu}(x,z,y)=\iint_{\Sn \times \Sn} c_\alpha (x,y)\, \dd \pi_{\mu,\nu} (x,y).
\end{split}
\end{equation}
The inequality \eqref{eq:lb-2} is saturated if and only if $z=p_{\alpha}(x,y)$ for $\pi_{\mu,\rho,\nu}$-a.e. $(x,z,y) \in \left(\Sn\right)^3$, that is, if $\rho=\rho_{\pi_{\mu, \nu}}^{(\alpha)}$. Moreover, the right-hand side of \eqref{eq:lb-2} is minimal if and only if $\pi_{\mu,\nu}=\pi_{*}$. Consequently,
\[
Q_{\alpha}^{\mu, \nu}(\rho) \geq \min_{\pi \in \Pi(\mu, \nu)} \lers{\iint_{\Sn \times \Sn} c_\alpha(x,y)\, \dd \pi(x,y)}=\iint_{\Sn \times \Sn} c_\alpha(x,y)\, \dd \pi_*(x,y)
\]
and the only $\rho$ realizing this minimum is $\rho_{\pi_{*}}^{(\alpha)}$. On the other hand, if $\pi_{*}$ puts weight on antipodal points, that is, $\pi_*\ler{\lers{(z,-z) : z\in \mathbb{S}^n}}>0$, and $\alpha=\frac{1}{2}$, then we have an infinite collection of minimizing measures by the theorems of Thales and Pythagoras --- or by a simple direct computation.
\end{proof}

In the next proposition, we consider the case when the first argument of $p_\alpha$ is fixed and clarify injectivity and surjectivity properties of $p_\alpha$ as $\alpha$ varies from 0 to 1.

\begin{proposition} \label{prop:inj-surj}
    Let $\alpha\in\left(0,\frac{1}{2}\right)\cup\left(\frac{1}{2},1\right]$ and let $N \in \mathbb{S}^n$ be arbitrary but fixed  --- it may be considered as the ``north pole''. Let $p_\alpha(N, \cdot)\colon \mathbb{S}^n\to \mathbb{S}^n$ be the map sending $u$ to
    \[
    p_\alpha(N,u)=\frac{(1-\alpha)N+\alpha u}{\|(1-\alpha)N+\alpha u\|}.
    \]
    Then for $\alpha\in\left(\frac{1}{2},1\right]$, $p_\alpha(N, \cdot)$ is bijective.
    For $\alpha\in\left(0,\frac{1}{2}\right)$, $p_\alpha(N, \cdot)$ is neither surjective nor injective: it is 2-to-1 for almost all points of $\mathbb{S}^n$.
    Finally, $p_{\frac{1}{2}}(N, \cdot)\colon \mathbb{S}^n\setminus\{-N\}\to \mathbb{S}^n\setminus\{-N\}$ is injective, and its range is the open ``upper'' hemisphere $\lers{z \in \mathbb{S}^n \, \middle| \, \inner{z}{N}>0}$.
\end{proposition}
\begin{proof}
When considering $p_\alpha(N,u)$, we can assume without loss of generality that $N=(0,0, \dots, 0,1)$ and $u=(\cos \theta, 0, \dots, 0, \sin \theta)$ for some $\theta \in (-\pi, \pi]$.
    Let $c_u^{(\alpha)}$ be the normalising constant $\|(1-\alpha)N+\alpha u\|$. Note that $c_u^{(\alpha)}>0$ if and only if $(\alpha,u)\neq\left(\frac{1}{2},-N\right)$. Whenever $(\alpha,u)\neq\left(\frac{1}{2},-N\right)$, we have that
    \[
    c_u^{(\alpha)}p_u(N,u)=(\alpha\cos\theta, 0, \dots, 0,(1-\alpha)+\alpha\sin\theta),
    \]
    and so for a fixed $\alpha\neq\frac{1}{2}$, setting
    \begin{align*}
        x_\theta&:=\alpha\cos\theta\\
        y_\theta&:=(1-\alpha)+\alpha\sin\theta,
    \end{align*}
    we see that $(x_\theta,y_\theta)$ satisfy
    $x_\theta^2+(y_\theta-(1-\alpha))^2=\alpha^2$, i.e., they lie on the circle of radius $\alpha$ with the centre at $(0,1-\alpha)$.
    For any point $u=(\cos\theta,0,\dots, 0,\sin\theta)$, its image $p_\alpha(N,u)$ is the projection of $(x_\theta,0, \dots,0, y_\theta)$ onto $\mathbb{S}^n$, i.e., the point obtained as the intersection of $\mathbb{S}^n$ and the half-line from $(0,\dots,0)$ through $(x_\theta,0, \dots,0,y_\theta)$. Now the statements of this proposition are easy to see from Figure \ref{figure:inj-surj}.
  \begin{figure}[H]
    \centering
    \includegraphics[width=\textwidth]{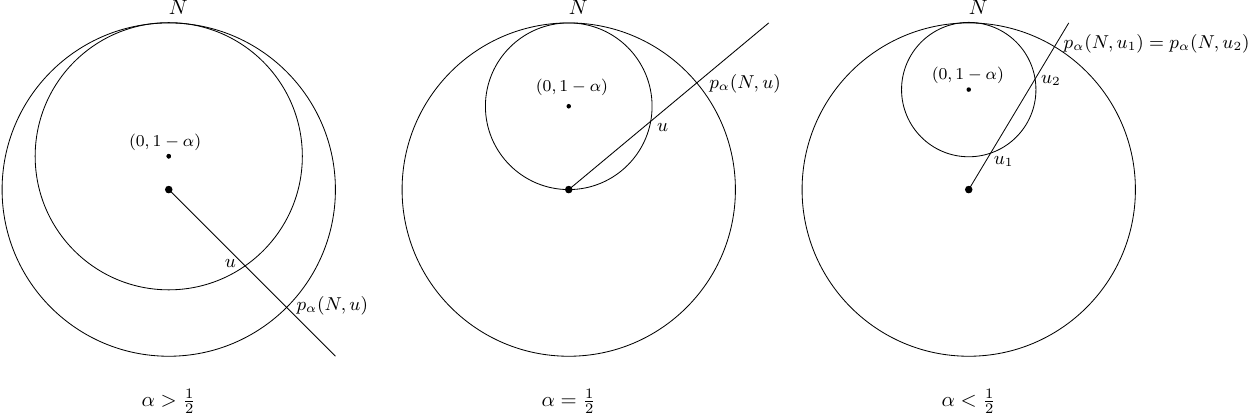}
    \caption{$p_\alpha(N,u)$ lies on the $\mathbb{S}^1$ spanned by $N$ and $u$, at the spherical projection of $c_u^{(\alpha)}p_\alpha(N,u)$. The bigger circle displayed in each of the cases is $\mathbb{S}^1$, the smaller one is $c_{\ \cdot}^{(\alpha)}p_\alpha(N,\cdot)$.}
    \label{figure:inj-surj}
\end{figure}
\end{proof}

Now we turn to the proof of Theorem \ref{thm:main} which, for the sake of clarity, we divide into six steps.\\

\noindent{\emph{Step 1.}} Similarly as in the proof of Theorem \ref{thm:circle}, we first understand the action of $\Psi$ on the set of Dirac measures. The maximal distance in $\mathcal{W}_2\left(\Sn,\|\cdot\|\right)$ is $2$, and is attained only on pairs of Dirac measures that are concentrated on antipodal points. Since
\[2=\dwt(\mu,\nu)=\dwt(\Psi(\mu),\Psi(\nu)),\] we get that $\Psi(\delta_x)$ is a Dirac measure for all $x\in\Sn$. Since $\Psi$ and $\Psi^{-1}$ are both isometries, the map $\psi\colon\Sn\to\Sn$ defined by $\Psi(\delta_x)=\delta_{\psi(x)}$ is a bijection, and furthermore, since
\[
\|\psi(x)-\psi(y)\|=d_{\mathcal{W}_2}\ler{\delta_{\psi(x)},\delta_{\psi(y)}}
=d_{\mathcal{W}_2}\ler{\Psi(\delta_x),\Psi(\delta_y)}
=d_{\mathcal{W}_2}\ler{\delta_x,\delta_y}=\|x-y\|,
\]
it is in fact an isometry.
Just as before, we will be done once we prove that an isometry of $\Wtsn$ which fixes all Dirac measures is necessarily the identity, because then in particular, $\ler{\psi^{-1}}_\#\circ\Psi=\id_{\Wtsn}$, and so $\Psi=\psi_\#$ as claimed. From now on, we assume that $\Phi$ is an isometry of $\Wtsn$ such that $\Phi(\delta_x)=\delta_x$, and our aim is to show that $\Phi(\mu)=\mu$ for all $\mu\in\Wtsn$.\\

\noindent{\emph{Step 2.}} Next we claim that $\Phi$ preserves the barycenter of measures.

For any $\mu\in\Wtsr$ and $x\in\Sn$, we have
\be \label{eq:dirac-dist-exp}
\dwt^2(\mu,\delta_x) = \int_{\Sn} \|y-x\|^2\dd\mu(y)=2-2\inner{x}{\int_{\Sn} y \, \dd\mu(y)}=2(1-\inner{x}{m(\mu)}).
\ee

This implies that
\[
2(1-\inner{x}{m(\mu)})=\dwt^2\ler{\mu,\delta_x}=\dwt^2\ler{\Phi(\mu),\Phi(\delta_x)}=\dwt^2\ler{\Phi(\mu),\delta_x}=2\ler{1-\inner{x}{m(\Phi(\mu))}}.
\]
But then
\[
\inner{x}{m(\Phi(\mu))}=\inner{x}{m(\mu)}
\]
and hence
\begin{equation}\label{eq:inner}
\inner{x}{m(\Phi(\mu))-m(\mu)}=0.
\end{equation}
Since for a fixed $\mu$, equation (\ref{eq:inner}) holds for every $x\in\Sn$, we conclude that $m\ler{\Phi\ler{\mu}} = m\ler{\mu}$.\\

\noindent{\emph{Step 3.}} Now we prove that measures supported on two points are mapped to measures supported on two points. We first show that for all $\mu,\nu\in\Wtsr$,

\[\affspan\ler{\supp\big(\Phi(\mu)\big)} \perp \affspan\ler{\supp\big(\Phi(\nu)\big)}\]
holds if and only if 
\[\affspan\ler{\supp(\mu)} \perp \affspan\ler{\supp(\nu)}.\]

Kloeckner proved in \cite[Lemma 6.2]{Kloeckner-2010} that orthogonality of supports can be characterized by the metric in the ambient space $\W_2\ler{\Rn, \|\cdot\|}$. Namely,
\begin{equation*}\label{eq:ort}
d_{\W_2(\Rn)}^2(\mu,\nu) = \|m(\mu)-m(\nu)\|^2+d_{\W_2(\Rn)}^2\left(\mu,\delta_{m(\mu)}\right)+d_{\W_2(\Rn)}^2\left(\nu,\delta_{m(\nu)}\right)
\end{equation*}
holds if and only if there exist two orthogonal affine subspaces $L,M \subset \Rn$ such that $\supp(\mu)\subseteq L$ and $\supp(\nu)\subseteq M$.
We proceed by showing that the isometries of $\Wtsn$ leave the $\W_2(\Rn)$-distance of a measure from the Dirac mass concentrated on its barycenter invariant, that is,
$$d_{\W_2(\Rn)}\ler{\mu,\delta_{m(\mu)}}=d_{\W_2(\Rn)}\ler{\Phi(\mu),\delta_{m(\Phi(\mu))}}$$ for any $\mu\in\Wtsn$. Indeed, a direct computation very similar to \eqref{eq:dirac-dist-exp} shows that
\[
d_{\W_2(\Rn)}^2\ler{\mu,\delta_{m(\mu)}}=1-\|m(\mu)\|^2 \text{ and } d_{\W_2(\Rn)}^2\ler{\Phi(\mu),\delta_{m(\Phi(\mu))}}=1-\|m(\Phi(\mu))\|^2,
\]
which implies our statement as we have shown $m(\Phi(\mu))=m(\mu)$ in Step 2. Hence for any $\mu, \nu\in\Wtsn$,
\begin{multline*}
\|m(\mu)-m(\nu)\|^2+\dwt^2\left(\mu,\delta_{m(\mu)}\right)+\dwt^2\left(\nu,\delta_{m(\nu)}\right)\\
=\|m(\Phi(\mu))-m(\Phi(\nu))\|^2+\dwt^2\left(\Phi(\mu),\delta_{m(\Phi(\mu))}\right)+\dwt^2\left(\Phi(\nu),\delta_{m(\Phi(\nu))}\right),
\end{multline*}
meaning that orthogonally supported measures must be mapped to orthogonally supported measures by $\Phi$. 

A maximal set of measures whose supports are one-dimensional and pairwise orthogonal must therefore be mapped to a set of measures whose supports are zero- or one-dimensional. But zero-dimensionally supported measures are exactly the Dirac masses, to which only Dirac masses can be mapped by $\Phi$, and so one-dimensionally supported measures must be mapped to one-dimensionally supported measures. Continuing similarly, we would see more generally that the affine dimension of the support is preserved by $\Phi$, but since on the sphere, one-dimensionally supported measures are exactly the two-point supported measures, the one-dimensional case is enough to prove our statement.\\

\noindent\emph{Step 4.} We proceed with showing that measures supported on two points are fixed by $\Phi$.
Let us introduce the notation $\Delta_{2}'(\Sn)$ for the set of all elements in $\Wtsr$ with a two-point support, set $\widetilde{\mu} := \ler{t_{-m(\mu)}}_\#\mu$ for all $\mu \in \Delta_{2}'(\Sn)$, and $\Delta_{2,0}'(\Sn) := \left\{\widetilde{\mu}\in\mathcal{P}(\Rn) : \mu\in \Delta_{2}'(\Sn) \right\}$. By Step 3, $\Phi|_{\Delta_{2}'(\Sn)}\colon \Delta_{2}'(\Sn)\to \Delta_{2}'(\Sn)$ is an isometric embedding. By Proposition \ref{prop:translation} we know that for all $\mu,\nu\in \Delta_{2}'(\Sn)$,
\begin{align*}
\dwt^2\ler{\widetilde{\mu},\widetilde{\nu}}
&= \dwt^2\ler{\mu,\nu} - \|m(\mu)-m(\nu)\|^2\\ 
&= \dwt^2\ler{\Phi(\mu),\Phi(\nu)} - \|m(\Phi(\mu))-m(\Phi(\nu))\|^2= \dwt^2\ler{\widetilde{\Phi(\mu)},\widetilde{\Phi(\nu)}}.
\end{align*}
Consequently, $\widetilde{\mu}=\widetilde{\nu}$ holds if and only if $\widetilde{\Phi(\mu)}=\widetilde{\Phi(\nu)}$, in other words, $\Phi(\nu)$ is a translate of $\Phi(\mu)$ if and only if $\nu$ is a translate of $\mu$. 

Let a measure $\mu\in\Delta_{2}'(\Sn)$ be fixed. We can assume without loss of generality that
\[
\supp(\mu)= \{(\cos\theta, 0, \dots,0, \sin \theta),(\cos\theta, 0, \dots,0, -\sin \theta)\}
\]
for some $\theta \in (0, \pi/2]$. In this case,
\[
\fel\sum_{x\in\supp(\mu)}x=(\cos \theta, 0, \dots, 0)
\]
and
\[
\ler{\affspan\ler{\supp(\mu)}-\fel\sum_{x\in\supp(\mu)}x}^\perp=\{(v_1,\dots,v_n,0) \colon v_1, \dots,v_n \in \R\}.
\]
 Define
 $\vec{\mu} :=\left\{v\in\R^{n+1} : (t_v)_\#\mu\in\mathcal{P}\ler{\Sn}\right\}$, and observe that $\vec{\mu}$ is the set of those vectors $(v_1,\dots,v_{n+1})$ such that
\begin{align*}\|(v_1+\cos \theta, v_2, \dots, v_n, v_{n+1}+\sin \theta)\|
=\|(v_1+\cos \theta, v_2, \dots, v_n, v_{n+1}-\sin \theta)\|=1.
\end{align*}
Then we get that $v\in\vec{\mu}$ if and only if
\[
(v_1+\cos\theta)^2+v_2^2+\dots+v_n^2+v_{n+1}^2\pm2v_{n+1}\sin\theta=1-\sin^2\theta.
\]
Since $\sin\theta\neq0$, this holds exactly when $v_{n+1}=0$ and $(v_1+\cos\theta)^2+v_2^2+\dots+v_n^2=\cos^2\theta$, i.e., the first $n$ coordinates span an $n-1$-dimensional sphere with radius $|\cos\theta|$ centered at $\ler{-\cos\theta,0,\dots,0}$, or they are just the singleton containing $0 \in \R^n$ in the case $\cos\theta=0$.
In other words,
\begin{align*}
\vec{\mu}&=-(\cos\theta,0,\dots,0)+|\cos\theta|\cdot\ler{\Sn\cap\lers{\ler{v_1,\dots,v_n,0} : v_1,\dots,v_n\in\R}}\\
&=-\ler{\fel\sum_{x\in\supp(\mu)}x} + \|\fel\sum_{x\in\supp(\mu)}x\|\cdot\ler{\Sn\cap\ler{\affspan\ler{\supp(\mu)}-\fel\sum_{x\in\supp(\mu)}x}^\perp}.
\end{align*}

As $\Phi$ maps the translates of $\mu$ to the translates of $\Phi(\mu)$, there is an $\eta \in \Delta_{2,0}'(\Sn)\subset\mathcal{P}\ler{\Rn}$ such that
\begin{equation}\label{eq:translates}
\Phi\ler{(t_v)_\#\mu} = (t_{v+m(\mu)})_\#\eta \qquad (v\in\vec{\mu}).
\end{equation}
We emphasize that $\eta$ does not depend on $v$.
It follows that $\vec{\mu}+m(\mu)+\supp(\eta) \subset \Sn$. But by plugging $v=0 \in \R^{n+1}$ to (\ref{eq:translates}), we get that $\supp\ler{\Phi(\mu)}=m(\mu)+\supp(\eta)$, and so the previous line becomes $\vec{\mu}+\supp\ler{\Phi(\mu)}\subset\Sn$. By the definition of $\vec{\Phi(\mu)}$, this means that $\vec{\mu}\subseteq\vec{\Phi(\mu)}$.

For any $\mu$ with $\diam(\supp(\mu))<2$, we have that $\cos\theta\neq0$, and so $\vec{\mu}$ is an $n-1$-dimensional sphere, implying that $\vec{\mu}=\vec{\Phi(\mu)}$ and $\supp(\mu)=\supp\ler{\Phi(\mu)}$. Now $\mu$ and $\Phi(\mu)$ are probability measures with the same 2-point support and the same barycenter, and so $\mu=\Phi(\mu)$. Finally, $\mu=\Phi(\mu)$ for all $\mu\in\Delta'_2\ler{\Sn}$ by continuity of $\Phi$.\\

\noindent\emph{Step 5.}  Now assume that $\mu=\sum_{i=1}^m\lambda_i\delta_{x_i}$ where $x_i\neq-x_j$ for all $1\leq i<j\leq m$. Such measures form a dense subset of $\Wtsn$. We claim that
\[\supp\big(\Phi(\mu)\big)\subseteq\{x_1,\dots,x_m\}\cup\{-x_1,\dots,-x_m\}\]
and $\big(\Phi(\mu)\big)(\{x_i,-x_i\})=\mu(\{x_i\})$ for all $1\leq i\leq m$.

The proof of this claim relies on preserving the mass of \emph{bisectors} which are defined as follows: for $u,v \in \Sn$, the corresponding bisector is
\[
B(u,v):=\left\{y\in\Sn : \|u-y\|=\|v-y\|\right\}\cong\mathbb{S}^{n-1}.
\]
To start, we apply Lemma 3.17 from \cite{HIL} with $E=\R^{n+1}$, $p=2$, $x\in\Sn$, $a=1$ and $b=-1$ to obtain that
\begin{multline*}
\mu\ler{B(x,-x)}=\max\lers{\alpha : d_{\mathcal{W}_2}\ler{\mu,\alpha\delta_x+(1-\alpha)\delta_{-x}}=m_\mu}\\
-\min\lers{\alpha : d_{\mathcal{W}_2}\ler{\mu,\alpha\delta_x+(1-\alpha)\delta_{-x}}=m_\mu},
\end{multline*}
where $m_\mu:=\min\lers{d_{\mathcal{W}_2}\ler{\mu,\alpha\delta_x+(1-\alpha)\delta_{-x}} : 0\leq\alpha\leq1}$
and $B(x,-x)\cong\mathbb{S}^{n-1}$ is the bisector between $x$ and $-x$, i.e., the set of all points equidistant from $x$ and $-x$. But since $\Phi\ler{\alpha\delta_x+(1-\alpha)\delta_{-x}}=\alpha\delta_x+(1-\alpha)\delta_{-x}$ for all $\alpha\in[0,1]$ by Steps 1 and 4, we get that $m_\mu=m_{\Phi(\mu)}$, and subsequently $\mu(B(x,-x))=\ler{\Phi(\mu)}(B(x,-x))$ for all $x\in\Sn$.
Since for every $x\in\Sn$, $B(x,-x)$ is an $n-1$-dimensional subsphere of $\Sn$, and every $n-1$-dimensional subsphere of $\Sn$ is of the form $B(x,-x)$ for some $x\in\Sn$, the previous sentence says that $\mu(S)=(\Phi(\mu))(S)$ for every subsphere $S$ of codimension 1.

For every $\widetilde{x}\in\{x_1,\dots,x_m\}=\supp(\mu)$, there exists a sequence $\left(S_j\right)_{j\in\mathbb{N}}$ of $n-1$-dimensional subspheres of $\Sn$ such that $S_j\cap\supp(\mu)=\{\widetilde{x}\}$ for every $j$, and the intersection of any $n$ subspheres is trivial, that is, $\bigcap_{k=1}^n S_{j_k}=\{\tilde{x}, -\tilde{x}\}$
for any choice of $j_1<j_2<\dots<j_n$. Therefore,
\[\mu(\{\widetilde{x}\})=\mu\ler{S_j}=\big(\Phi(\mu)\big)(S_j) \qquad \ler{j \in \N},\]
and we are in the right position to prove that 
\[
\Phi(\mu)\ler{\{\widetilde{x},-\widetilde{x}\}}=\mu\ler{\{\widetilde{x}\}}.
\]
The inequality $\Phi(\mu)\ler{\{\widetilde{x},-\widetilde{x}\}} \leq \big(\Phi(\mu)\big)(S_j)=\mu(\{\widetilde{x}\})$ holds because $\lers{\widetilde{x},-\widetilde{x}}\subseteq S_j$. If $n=1$, then in fact $\lers{\Tilde{x},-\Tilde{x}}=S_j$, and we are done. If $n\geq2$, assume
indirectly that $\Phi(\mu)\ler{\{\widetilde{x},-\widetilde{x}\}}<\mu\ler{\{\widetilde{x}\}}$, and let $\varepsilon>0$ denote the gap between the two sides of this strict inequality. Now we have
\[
\big(\Phi(\mu)\big)\ler{S_j\setminus \{\widetilde{x},-\widetilde{x}\}}
=\big(\Phi(\mu)\big)(S_j)-\Phi(\mu)\ler{\{\widetilde{x},-\widetilde{x}\}}=\varepsilon
\]
for every $j \in \N$. The fact that $\bigcap_{k=1}^n S_{j_k}=\{\tilde{x}, -\tilde{x}\}$
for every $j_1<j_2<\dots<j_n$ implies that the family of sets $\ler{S_j\setminus \{\widetilde{x},-\widetilde{x}\}}_{j=1}^{\infty}$ covers any point of $\Sn$ at most $n$ times. This means that $\sum_{j=1}^{\infty} \big(\Phi(\mu)\big)\ler{S_j\setminus \{\widetilde{x},-\widetilde{x}\}}$ is bounded from above by $n \cdot \big(\Phi(\mu)\big)(\Sn)=n$, which is a contradiction as $\big(\Phi(\mu)\big)\ler{S_j\setminus \{\widetilde{x},-\widetilde{x}\}}=\varepsilon$ for every $j \in \N$ and $\sum_{j=1}^{\infty} \varepsilon= \infty$.
\\

\noindent\emph{Step 6.} A crucial consequence of the claim made in Step 5 is that the isometry $\Phi$ fixes all measures that are supported within an open hemisphere of $\mathbb{S}^n$. Indeed, we learned from Step 2 that $\Phi$ preserves the barycenter of measures, and from Step 5 that the only possible action $\Phi$ can do is to send some mass from a point to its antipodal point. But if a measure is supported on an open hemisphere, then the transport of any mass to its antipodal point would change the barycenter.
\\
Suppose that $\mu\ler{\{-N\}}=0$. Then the spherical projection $\rho^{(\frac{1}{2})}_{\delta_N\otimes\mu}$ of the displacement convex combination of $\delta_N$ and $\mu$ is well-defined, and since it is supported on the upper hemisphere, $\Phi\ler{\rho^{(\frac{1}{2})}_{\delta_N\otimes\mu}}=\rho^{(\frac{1}{2})}_{\delta_N\otimes\mu}$.
Let us now consider the sets
\begin{align*}
        A:&=\lers{Q_{\frac{1}{2}}^{\delta_N,\mu}(\rho) : \rho\in\mathcal{P}\left(\Sn\right)}\\
        &=\lers{Q_{\frac{1}{2}}^{\Phi(\delta_N),\Phi(\mu)}\ler{\Phi(\rho)} : \rho\in\mathcal{P}\left(\Sn\right)}\\
        &=\lers{Q_{\frac{1}{2}}^{\delta_N,\Phi(\mu)}(\kappa) : \kappa\in\mathcal{P}\left(\Sn\right)}=:B
\end{align*}
where the last equality follows from the surjectivity of the isometry $\Phi$. Since $A=B$, necessarily
    \[
    \min B = \min A = Q_{\frac{1}{2}}^{\delta_N,\mu}\ler{\rho^{(\frac{1}{2})}_{\delta_N\otimes\mu}}
    = Q_{\frac{1}{2}}^{\Phi(\delta_N),\Phi(\mu)}\ler{\Phi\ler{\rho^{(\frac{1}{2})}_{\delta_N\otimes\mu}}}
    = Q_{\frac{1}{2}}^{\delta_N,\Phi(\mu)}\ler{\Phi\ler{\rho^{(\frac{1}{2})}_{\delta_N\otimes\mu}}}.
    \]
    The fact that $B$ has a unique minimizer implies by the second statement of Proposition \ref{prop:min-char} that $\Phi(\mu)(\{-N\})=0$.
    Consequently --- let us now use the first statement of Proposition \ref{prop:min-char} ---, the unique minimizer of $Q_\alpha^{\delta_N,\Phi(\mu)}$ is $\rho^{(\frac{1}{2})}_{\delta_N\otimes\Phi(\mu)}$, and hence
      \be \label{eq:fineq}      \rho^{(\frac{1}{2})}_{\delta_N\otimes\Phi(\mu)}=\Phi\ler{\rho^{(\frac{1}{2})}_{\delta_N\otimes\mu}}=\rho^{(\frac{1}{2})}_{\delta_N\otimes\mu}.
    \ee
     By injectivity of $p_{\frac{1}{2}}(N,\cdot)$ on $\mathbb{S}^n\setminus \{-N\}$, see Proposition \ref{prop:inj-surj}, for every measure $\nu\in\mathcal{P}\ler{\Sn}$ supported within the upper hemisphere, there is a unique measure $\kappa\in\mathcal{P}\ler{\Sn}$ such that $\nu$ is the spherical projection $\rho^{(\frac{1}{2})}_{\delta_N\otimes\kappa}$ of the displacement convex combination of $\delta_N$ and $\kappa$. Therefore, \eqref{eq:fineq} implies that $\Phi(\mu)=\mu$, which completes the proof.\\

\paragraph*{{\bf Acknowledgements.}} We are grateful to the anonymous referee for his/her insightful comments and suggestions. 

\end{document}